\documentclass[11pt,reqno,oneside]{amsart}
\usepackage{amsthm,amsmath,amssymb,amscd,thmtools}
\usepackage{fullpage}
\usepackage{graphicx} 
\usepackage{subfig}

\usepackage[a4paper, total={6in, 9.1in}]{geometry}
\usepackage{hyperref}
\hypersetup{
  colorlinks,
  linkcolor={blue!80!black},
  urlcolor={blue!80!black},
  citecolor={blue!80!black}
}
\usepackage[capitalize,noabbrev]{cleveref}
\usepackage{amsfonts, mathrsfs}
\usepackage{xcolor}
\usepackage[normalem]{ulem}

\usepackage[utf8]{inputenc}
\usepackage{latexsym}
\usepackage[all]{xy}
\usepackage{mathtools}

\newtheorem{theorem}{Theorem}[section]

\newtheorem{lemma}[theorem]{Lemma}

\newtheorem{condition}[theorem]{Condition}

\theoremstyle{definition}
\newtheorem{remark}[theorem]{Remark}
\newtheorem{definition}[theorem]{Definition}

\numberwithin{equation}{section}
\title{Stochastic Moving Anchor Algorithms and a Popov's Scheme with Moving Anchor}
\author{James K. Alcala}
\thanks{James K. Alcala was partially supported by a UCR Dissertation Year Program Award.}
\address{Department of Mathematics, University of California, Riverside, USA}
\email{jalca014@ucr.edu}

\author{Yat Tin Chow}
\thanks{Yat Tin Chow is partially supported by a Regents' Faculty Fellowship at the University of California, Riverside, and by NSF DMS-2409903 and
ONR N00014-24-S-B001.}
\address{Department of Mathematics, University of California, Riverside, USA}
\email{yattinc@ucr.edu}

\author{Mahesh Sunkula}
\address{Department of Mathematics, Purdue University, West Lafayette, Indiana, USA}
\email{msunkula@purdue.edu}

\begin{document}

\maketitle
\begin{abstract}
    Since their introduction, anchoring methods in extragradient-type saddlepoint problems have inspired a flurry of research due to their ability to provide order-optimal rates of accelerated convergence in very general problem settings. Such guarantees are especially important as researchers consider problems in artificial intelligence (AI) and machine learning (ML), where large problem sizes demand immense computational power. Much of the more recent works explore theoretical aspects of this new acceleration framework, connecting it to existing methods and order-optimal convergence rates from the literature. However, in practice introducing stochastic oracles allows for more computational efficiency given the size of many modern optimization problems. To this end, this work provides the moving anchor variants \cite{alcala2023moving} of the original anchoring algorithms \cite{yoon_ryu} with stochastic implementations and robust analyses to bridge the gap from deterministic to stochastic algorithm settings. In particular, we demonstrate that an accelerated convergence rate theory for stochastic oracles also exists for our moving anchor scheme, itself a generalization of the original fixed anchor algorithms, and provide numerical results that validate our theoretical findings. We also develop a tentative moving anchor Popov scheme based on the work in \cite{tran2021halpern}, with promising numerical results pointing towards an as-of-yet uncovered general convergence theory for such methods.
\end{abstract}

\section{Introduction}

Saddle point problems of the form
\begin{align}
    \min_{x\in \mathbb{R}^{n}}\max_{y\in \mathbb{R}^{m}} f(x,y) \label{problem_setting},
\end{align}
also known as minimax or min-max problems, continue to be the object of intense study amongst researchers in a wide variety of disciplines.
Applications of such problems classically include economics and game theory, with more recent applications including mean field games \cite{nurbekyan2024monotoneinclusionmethodsclass}, as well as artificial intelligence and machine learning applications such as generative adversarial nets \cite{chavdarova2019reducing}, \cite{goodfellow2014generative} and reinforcement learning \cite{du2017stochastic}.
Recently, researchers have explored the Halpern iteration \cite{halpern1967fixed}, \cite{lieder2021convergence} as an acceleration mechanism in monotone inclusion problems \cite{diakonikolas2020halpern} and for minimax problems more specifically with a Halpern-inspired method known as anchoring \cite{ryu2019ode}.
This latter line of inquiry has proved especially fruitful, as the well-established extragradient method \cite{korpelevich1976extragradient} combined with anchoring yields an accelerated $\Omega(1/k^{2})$ convergence rate on the squared gradient norm for smooth-structured convex-concave minimax problems - this remarkable order-optimal result is thanks to the introduction of the extra-anchored gradient, or EAG, method and its variants \cite{yoon_ryu}.

This anchoring mechanism has been applied, for example, to develop an anchored Popov's scheme \cite{popov1980modification} and a splitting version of EAG \cite{tran2021halpern}, and has recently been connected \cite{tran2022connection} to Nesterov's classical Accelerated Gradient Method, or AGM \cite{nesterov1983method_AGM}.
Another important work \cite{lee2021fast} extended the initial results of the EAG methods to problem settings involving negative comonotone operators via the similar fast extra gradient algorithm, or FEG, which extends this fast acceleration rate to certain nonconvex-nonconcave problems.
This latter work also introduced the first stochastic anchored algorithm, which to our knowledge was the first instance of stochastic oracles involving anchored algorithms.
These authors also introduced the `semi'-anchoring \cite{lee2021semi} in a multi-step descent/ascent framework with a unique anchor occurring at each step of the multi-step, a generalization of the initial anchoring mechanism that relies solely on the initial point to be the anchor.

In this vein, \cite{alcala2023moving} recently developed variants of EAG \cite{yoon_ryu} and FEG \cite{lee2021fast} algorithms by introducing a moving anchor into both frameworks.
These `moving anchor' algorithms retain order-optimal convergence rates on the squared gradient norm in both convex-concave and the negative comonotone settings, and generalize the analyses of the previous works to acommodate this new moving anchor structure.
Most notably, the numerical experiments featured in this work show that moving anchor variants of EAG and FEG algorithms are faster than their fixed-anchor counterparts by a constant across numerous problem settings.
However, the strengths of the moving anchor algorithms over their fixed-anchor cousins are poorly understood, and most numerical examples are deterministic and toy examples.

To wit, the contributions of this paper are as follows.
\begin{enumerate}
    \item Stochastic moving anchor algorithms are developed following \cite{alcala2023moving}. Specifically, stochastic EAG moving anchor algorithms are defined and analyzed via a conventional Lyapunov functional analysis. These analyses follow in the footsteps of the stochastic complexity results of \cite{lee2021fast} but generalize to the moving anchor setting. Convergence results follow with minimal assumptions tacked on to the deterministic algorithm assumptions in \cite{alcala2023moving}, and rely on control over variance terms. 
    \item Numerous numerical examples further compare stochastic moving anchor algorithms to their fixed-anchor counterparts to characterize further the nature of the moving anchor's convergence improvements.
    \item A moving anchor Popov's scheme with preliminary numerical results showing accelerated order-optimal convergence rates that are faster than their fixed anchor counterparts by a constant. The theory for these methods is underdeveloped, but our preliminary work suggests a robust underlying convergence theory that greatly generalizes the fixed anchor methods found in \cite{tran2021halpern}.
\end{enumerate}

\section{Literature Review}

\subsection{Stochastic Saddlepoint Algorithms}

Because of the computational advantages of stochastic optimization algorithms in modern optimization, the current literature on stochastic saddlepoint problems is deep and rich. 
For general saddlepoint algorithms, we restrict ourselves to a few recent and interesting results.
For the setting with decision-dependent distributions, focusing on certain fixed points of training enables the construction of powerful derivative-free algorithms \cite{wood2023stochastic}. 
It has also been shown that stochastic saddlepoint algorithms with guarantees on their `strong gap' (as opposed to their weak gap, a primal-dual gap taken in expectation over the sample space) avoid spurious convergence rate discrepancies on even simple problems \cite{bassily2023differentially}.
Researchers also used the Maurey Sparsification Lemma to obtain a stochastic saddlepoint algorithm in the polyhedral setting that isn't based on Frank-Wolfe methods, the first of its kind \cite{gonzalez2024mirror}.

The setting of stochastic extragradient algorithms is also particularly well-studied, and is of interest to the authors of this paper. 
In \cite{gorbunov2022stochastic}, a general framework is developed to study and prove results for a variety of specific stochastic extragradient methods.
The authors of \cite{li2022convergence} show that in the bilinear problem setting when equipped with iteration averaging and restarting, stochastic extragradient goes beyond converging to a fixed neighborhood of the Nash equilibrium solution, and eventually reaches the equilibrium point.
The last work we mention combines the celebrated Nesterov acceleration \cite{nesterov1983method_AGM} with extragradient, known as AG-EG \cite{yuan2023optimal}, brings optimal convergence rates for strongly monotone variational inequalities (of which saddlepoint problems are a special class), and even attains convergence rates matching lower bounds for bilinearly-coupled strongly-convex strongly-concave saddlepoint problems.

\subsection{Extragradient and Halpern Variants in Optimization}

The Halpern iteration \cite{halpern1967fixed} is an algorithm that finds fixed-points given a nonexpanding map, and has proven extremely fruitful within optimization as an acceleration method especially for saddlepoint algorithms \cite{diakonikolas2020halpern}, \cite{ryu2019ode}, \cite{tran2021halpern}, \cite{tran2022connection}. Anchoring \cite{yoon_ryu}, one particular application of the Halpern iteration, has made many waves in the field of saddlepoint problems distinct from Nesterov acceleration \cite{nesterov1983method_AGM}. Anchoring has been studied in the continuouos-time setting \cite{suh2023continuoustime}, exhibits a `merging path' property \cite{yoon2025flock}, has been applied reinforcement learning \cite{lee2023accelerating}, and has recently been shown to not be a unique mechanism in optimal acceleration \cite{yoon2024notunique}.

The extragradient algorithm \cite{korpelevich1976extragradient}, short for extrapolated gradient, has enjoyed similar popularity and interest within the optimization community \cite{azizian2020tight}, especially relating to generative adversarial networks \cite{liu2019towards}, \cite{chavdarova2019reducing}, and adversarial training \cite{madry2017towards}. 
Due to the popularity of these methods, there is a wealth of literature on different varieties of extragradient methods in both stochastic and deterministic settings. 
We briefly mention that in 2022, several researchers used a novel Lyapunov functional to obtain tight last-iterate convergence guarantees for extragradient \cite{cai2022tightlastiterateconvergenceextragradient} that are indeed order-optimal for the convergence rate on the gap function \cite{golowich}.
The synthesis of extragradient with anchoring \cite{yoon_ryu} introduced new order-optimal convergence rates on a different optimality measure, the squared gradient norm, which further motivated investigations into the synthesis of extragradient and extragradient-adjacent methods \cite{tran2021halpern} with anchoring, opening an exciting avenue of research for new accelerated methods.
The last article we mention brings further analysis of extragradient and the optimistic gradient to negative comonotone settings, expanding the reach of such methods to the nonconvex-nonconcave settings \cite{gorbunov2023convergence}.

\section{Preliminaries}

\subsection{Notation}

A saddle function \eqref{problem_setting} is convex-concave if it is convex in $x$ for any fixed $y\in\mathbb{R}^{m}$ and concave in $y$ for any fixed $x\in\mathbb{R}^{n}$.
A saddle point $(\hat{x}, \hat{y}) \in \mathbb{R}^{n} \times \mathbb{R}^{m}$ is any point such that the inequality $f(\hat{x}, y) \leq f(\hat{x}, \hat{y}) \leq f(x, \hat{y})$ for all $x \in \mathbb{R}^n$ and $y\in \mathbb{R}^{m}.$
Solutions to \eqref{problem_setting} are defined as saddle points. For this paper, we assume the differentiability of $f$, and we are especially interested in the so-called \textit{saddle operator} associated to $f$,
\begin{align} \label{sad_operator} G_{f}(z) &= \left[
    \begin{array}{c}
    \nabla_{x} f(x,y)\\
    -\nabla{y} f(x,y)
\end{array}
\right] \end{align}
where the $f$ subscript is omitted when the underlying saddle function is known.
When our problem is convex-concave, the operator \eqref{sad_operator} is monotone: $\langle G_{f}(z^{1}) - G_{f}(z^{2}),z^{1} - z^{2} \rangle \geq 0 \;\forall z_{1}, z_{2} \in \mathbb{R}^{n} \times \mathbb{R}^{m}.$
We assume that this operator $G_{f}$ is $R$-Lipschitz, or has certain stronger Lipschitz properties we detail later; this is sometimes referred to as $f$ being $R$-smooth.

The notation $\mathbb{E}[X|z]$ denotes the expectation of $X$ given $z$, which will for us generally be a vector.
Throughout this paper, the notation $z^{k}$ indicates the $k-$ iterate of some algorithm (stochastic or deterministic), and we use $z$ to represent a vector $(x,y) \in \mathbb{R}^{n} \times \mathbb{R}^{m}$.

\subsection{Deterministic Moving Anchor EAG-V}

In this section we detail the (explicit) moving anchor algorithms along with their convergence results and lemmas; further details on the proofs of these results may be found in \cite{alcala2023moving}.

The $k-th$ iterate of $z^{0} \in \mathbb{R}^{n} \times \mathbb{R}^{m}$ for the EAG-V with moving anchor is defined as 
\begin{align}
    \notag z^{0}       &= \bar{z}^{0}\\
    \label{eagv update1} z^{k+1/2}   &= z^{k} + \frac{1}{k+2}(\bar{z}^{k} - z^{k}) - \alpha_{k}G(z^{k})\\
    \label{eagv update2} z^{k+1}     &= z^{k} + \frac{1}{k+2}(\bar{z}^{k} - z^{k}) - \alpha_{k}G(z^{k + 1/2})\\
    \label{eagv update3} \bar{z}^{k+1}&= \bar{z}^{k} + \gamma_{k+1}G(z^{k+1})\\
    \notag \alpha_{k+1}&= \frac{\alpha_{k}}{1 - \alpha_{k}^{2}R^{2}}\left(1 - \frac{(k+2)^{2}}{(k+1)(k+3)}\alpha_{k}^{2}R^{2}\right)\\
    &= \alpha_{k}\left(1 - \frac{1}{(k+1)(k+3)}\frac{\alpha_{k}^{2}R^{2}}{1 - \alpha_{k}^{2}R^{2}}\right), \label{normal_alpha}
\end{align}

where $\alpha_{0} \in (0,1/R),$ and R is our Lipschitz constant. The following auxiliary sequences are important for numerics and the Lyapunov analysis:

\begin{align}
    \label{c_update_1} c_{k+1}     &= \frac{c_{k}}{1 + \delta_{k}},\\
    \label{gamma_update_1} \gamma_{k+1}&= \frac{B_{k+1}}{c_{k+1}(1 + \frac{1}{\delta_{k}})}.
\end{align}

We choose $\delta_{k}$ so that $\displaystyle \sum_{k=0}^{\infty} \log(1 + \delta_{k}) < \infty$. 
The $c_{k}$ terms are part of the definition of the Lyapunov functional we use in the analysis.
Let $c_{\infty}:= \lim_{k \to \infty} c_{k} = c_{0} \prod_{k=0}^{\infty} \frac{1}{1+\delta_{k}}$. 
One chooses $c_{0}$ so that $c_{\infty}$ satisfies some specified convergence constraint; these constraints will appear throughout the major convergence theorems in this section and the next section.
While the choice of $c_{0}$ is therefore limited to according to certain problem/algorithm constraints, in general there is freedom in choosing $c_{0}$ and the sequence $\{\delta_{k}\}.$
Furthermore, we generally take \eqref{c_update_1} and \eqref{gamma_update_1} to be given with equal signs instead of inequalities. 
For clarity, we emphasize that the original (fixed-anchor) EAG-V algorithm may be recovered simply by setting $\gamma_{k+1} := 0$ for all $k.$

First, we clarify the details about the sequence given in \eqref{normal_alpha}:
\begin{lemma}\label{normal_alpha_lemma}
    If $\alpha_{0} \in (0, \frac{3}{4R})$, then the sequence $\{\alpha_{k}\}_{k=0}^{\infty}$ of \eqref{normal_alpha} monotonically decreases to a positive limit.
\end{lemma}
\begin{proof}
    This is proved as a corollary of \cref{funny_alpha_lemma}.
\end{proof}

\begin{lemma}\label{LpnvLm_1}
    Let $\{\beta_{k}\}_{k=0}^{\infty} \subset (0,1)$ and $\alpha_{0} \in (0,\frac{1}{R})$ be given.
    Define the sequences $\{A_{k}\}_{k=0}^{\infty}$, $\{B_{k}\}_{k=0}^{\infty}$, $\{\alpha_{k}\}_{k=0}^{\infty}$ by the recurrences
    \begin{align}
        A_{k} & = \frac{\alpha_{k}}{2\beta_{k}}B_{k} \label{ordinary_A} \\
        B_{k+1} & = \frac{B_{k}}{1-\beta_{k}} \label{ordinary_B} \\
        \alpha_{k+1} & = \frac{\alpha_{k}\beta_{k+1}(1 - \alpha_{k}^{2}R^{2} - \beta_{k}^{2})}{\beta_{k}(1-\beta_{k})(1-\alpha_{k}^{2}R^{2})} \label{ordinary_alpha}
    \end{align}
    for $k\geq 0$ and $B_{0} = 1$.
    Suppose $\alpha_{k} \in (0,\frac{1}{R})$ for all $k\geq 0$ and $f$ is $R-$smooth and convex-concave.
    Then the Lyapunov functional
    \[
        V_{k} := A_{k}\|G(z^{k})\|^{2} + B_{k}\langle G(z^{k}),z^{k} - \bar{z}^{k}\rangle +  c_{k}\|z^{*} - \bar{z}^{k}\|^{2},
    \]
    for moving anchor EAG-V iterations \eqref{eagv update1} - \eqref{eagv update3} is nonincreasing.
\end{lemma}

\begin{remark} \label{Lpnv_RMK}
    In this section and throughout much of this paper, the notation developed regarding the sequences $A_{k}, \alpha_{k}, B_{k},$ and $\beta_{k}$ will err on the side of generality.
    However, in practice (specifically in \cite{yoon_ryu}, \cite{alcala2023moving}) a common choice of $\beta_{k}$ is $\frac{1}{k+2}$ which leads to $B_{k} = k+1,$ $A_{k} = \frac{\alpha_{k}(k+1)(k+2)}{2},$ and \eqref{normal_alpha}. Any differences needed later in these constants will be clarified.
\end{remark}

\begin{theorem} \label{conv_conc_convergence_thm}
    The EAG-V algorithm with moving anchor, described above, together with the Lyapunov functional described in \cref{LpnvLm_1}, has convergence rate
    \[
        \|G(z^{k})\|^{2}\leq \frac{4(\alpha_{0}R^{2} + c_{0})\|z^{0}-z^{*}\|^{2}}{\alpha_{\infty}(k+1)(k+2)}
    \]
    as long as we assume $c_{\infty}\alpha_{\infty} \geq 1.$
\end{theorem}

\cref{LpnvLm_1} and \cref{conv_conc_convergence_thm} are the primary result of \cite{alcala2023moving} for the deterministic moving anchor algorithms in the convex-concave EAG-V algorithm setting.

\subsection{The Anchored Popov's scheme}

One of the shortcomings of traditional extragradient methods such as those proposed in \cite{alcala2023moving}, \cite{yoon_ryu}, \cite{lee2021fast} is that they require at least two gradient evaluations at each time step.
(Even in \cite{alcala2023moving}, the evaluation at \cref{eagv update3} can be saved and used in \cref{eagv update1} at the next time step.)
In \cite{tran2021halpern}, a new variant of the original extra-anchored gradient algorithm \cite{yoon_ryu} appeared as a variant of the classical Popov's scheme \cite{popov1980modification}.
One advantage of Popov's scheme is that it requires only a single evaluation of the operator. 
This partially motivated the development of the anchored Popov's scheme, which we now detail.

Popov's original algorithm \cite{popov1980modification} has the following form, very similar to the original extragradient \cite{korpelevich1976extragradient}:

\begin{align*}
    \hat{z}^{k} &= z^{k} - \eta_{k}G(\hat{z}^{k-1}) \\
    z^{k+1} &= z^{k} - \eta_{k}G(\hat{z}^{k})
\end{align*}

The key difference between this and extragradient is that the directions computed via gradient are only ever evaluated at the points $\hat{z}^{k}$. The anchored Popov's scheme has the form 

\begin{align}
    \hat{z}^{k} &= \beta_{k}z^{0} + (1 - \beta_{k})z^{k} - \eta_{k}G(\hat{z}^{k-1}) \label{anch_pop_1}\\
    z^{k+1} &= \beta_{k}z^{0} + (1 - \beta_{k})z^{k} - \eta_{k}G(\hat{z}^{k}), \label{anch_pop_2}
\end{align}

where setting $\beta_{k} = 0$ brings back the original Popov's scheme. It is also worth pointing out that this classical method is equivalent to the optimistic gradient method used in online learning \cite{hsieh2019}.
The stepsize regime $\eta_{k}$ developed in \cite{tran2021halpern} is a slight variant of the one developed in \cite{yoon_ryu} but has more or less the same analysis.

We also point out that the anchored Popov scheme may be rewritten as 
\[
\begin{cases}
    \hat{z}^{k} &= (\beta_{k} - \frac{\beta_{k-1}\eta_{k}}{\eta_{k-1}})z^{0} + (1 - \beta_{k} + \frac{\eta_{k}}{\eta_{k-1}})z^{k} - \frac{(1 - \beta_{k-1})\eta_{k}}{\eta_{k-1}}x_{k-1}, \\
    z^{k+1} &= \beta_{k}z^{0} + (1 - \beta_{k})z^{k} - \eta_{k}G(\hat{z}^{k}),
\end{cases}
\]
which, for $\beta_{k} := 0$ and $\eta_{k} = \eta > 0$, reduces to $z^{k+1} = z^{k} - \eta G(2z^{k} - z^{k-1}),$ which is the \textit{reflected gradient method} proposed in \cite{Malitsky_2015}. This implies that the anchored Popov's scheme is an accelerated reflected gradient method.

Finally, we remark that this method has a similar Lyapunov analysis to other anchor methods \cite{yoon_ryu}, \cite{lee2021fast}, \cite{alcala2023moving} with a similar $O(1/k^{2})$ convergence rate guarantee on the squared gradient norm, though the anchored Popov's scheme analysis is somewhat more arduous as the operator evaluation doesn't depend on the previous iterate. In addition, there is some overestimation of constants in the demonstration of convergence, leading to a larger constant factor in the convergence bound, see Theorem 1, \cite{tran2021halpern}.

\section{The Stochastic Moving Anchor}


Let $G$ be an $R-$Lipschitz, monotone operator on $\mathbb{R}^{n} \times \mathbb{R}^{m}$, and let $N=n+m.$
To develop the stochastic moving anchor EAG-V algorithm, the following additional clarifications and assumptions are necessary.
\begin{enumerate}
    \item $\displaystyle \frac{1}{N} \sum_{i=1}^{N}G_{i}(z) = \mathbb{E}[G_{\theta}(z) | z] = G(z)$, or the expectation of $G_{\theta}(z)$ given $z$ is $G(z)$ for $\theta$ iid on ${1,\ldots,N}.$ \label{condition1}
    \item $G$ has condition number $\overline{C_{G}}(z)$, dependent on the point $z$ being evaluated by $G$, such that $\displaystyle \frac{1}{N} \sum_{i=1}^{N} \|G_{i}(z)\|^{2} \leq  \overline{C_{G}}(z)\|G(z)\|^{2}$ and $1 \leq\overline{C_{G}}(z) \leq N$ are true for all $z \in \mathbb{R}^{n} \times \mathbb{R}^{m}$, with the term $\overline{K_{G}}(z) := N\overline{C_{G}}(z)$ resulting in $1 \leq \overline{K_{G}}(z) \leq N^{2}.$
    Note that this also gives us an a priori bound on certain variance terms.
    Given three indices $i, j, k,$ (whose meaning will become apparent below), we have
    \begin{align}
        Var(z) &:= \mathbb{E}[\langle G(z) - G_{i_{j}^{k}}(z), G(z) - G_{i_{j}^{k}}(z)\rangle | z] \notag \\
        &= \frac{1}{N} \sum_{i=1}^{N} \|G_{i}(z)\|^{2} - \|G(z)\|^{2} \notag \\
        &\leq (\overline{C_{G}} - 1)\|G(z)\|^{2} \label{var_bound}
    \end{align} 
    no matter what values $i, j, k,$ may take.
    Note that the condition number, as defined, has this property \eqref{var_bound} that holds for any $z$; however, we are particularly interested in the behavior of $Var(z^{k}), Var(z^{k+1/2}),$ where $z^{k}$ is the $k-$th iteration of a stochastic algorithm and $z^{k+1/2}$ is a sort of extrapolation step.
    Therefore, we impose one other useful bound regarding this condition number $\overline{C_{G}}(z)$ as it relates to the stochastic iterates and half-iterates in the stochastic algorithm we define below.
    \begin{condition} \label{summability_condition}
    The term $\overline{C_{G}}(z)$ depends on the local value $z$ being evaluated by the operator $G$ in such a way that the following inequalities hold for all $k$, where $k$ is the iteration count of a stochastic algorithm:
    \begin{align}
        \big{(}\overline{C_{G}}(z^{k})-1\big{)}\|G(z^{k})\|^{2} &\leq \frac{C_{1}}{(k+1)^{4}} \label{summ_cond_1} \\
        \mathbb{E}[\big{(}\overline{C_{G}}(z^{k+1/2})-1\big{)}\|G(z^{k+1/2})\|^{2}|\bar{z}^{k},z^{k}] &\leq \frac{C_{2}}{(k+1)^{4}}, \label{summ_cond_2}
    \end{align}
    where $C_{1}, C_{2}$ are fixed positive constants.
    \end{condition}
    With \cref{summability_condition} in mind, we will henceforth use the notation $C_{G}(z^{k}), K_{G}(z^{k})$ to indicate two nonnegative, real-valued functions from $\mathbb{R}^{n} \times \mathbb{R}^{m}$ to $\mathbb{R}$ that behave according to \eqref{summ_cond_1}, \eqref{summ_cond_2}.
    In particular, we note that for any $z^{k}$ coming from a stochastic algorithm, we have that eventually,
    \begin{equation}
        C_{G}(z^{k}) \leq \overline{C_{G}},
    \end{equation}
    where we define $\bar{C_{G}}$ to be the supremum of such condition numbers independent of $z$.
    \label{condition2}
    \item For all $z_{1}, z_{2} \in \mathbb{R}^{n} \times \mathbb{R}^{m},$ $\displaystyle \|G_{i}(z_{1}) - G_{i}(z_{2})\|^{2} \leq R_{i}^{2}\|z_{1} - z_{2}\|^{2}$ with $R = \sqrt{\sum R_{i}}.$ \label{condition3}
\end{enumerate}
Furthermore, define $\{i_{k}^{1}, i_{k}^{2}, i_{k}^{3}\}_{k=1}^{\infty}$ to be uniformly iid random on $\{1, \ldots, N\}$.
Then the \textbf{stochastic EAG-V with moving anchor} is defined as 
\begin{align}
    \label{stoch eagv update1} z^{k+1/2}   &= z^{k} + \frac{1}{k+2}(\bar{z}^{k} - z^{k}) - \alpha_{k}G_{i_{k}^{1}}(z^{k})\\
    \label{stoch eagv update2} z^{k+1}     &= z^{k} + \frac{1}{k+2}(\bar{z}^{k} - z^{k}) - \alpha_{k}G_{i_{k}^{2}}(z^{k + 1/2})\\
    \label{stoch eagv update3} \bar{z}^{k+1}&= \bar{z}^{k} + \Tilde{\gamma}_{k+1}G_{i_{k}^{3}}(z^{k+1}) \\
    \label{gamma_update_2} \Tilde{\gamma}_{k+1} &= \frac{B_{k+1}}{c_{k+1}\overline{K_{G}}(1 + \frac{1}{\delta_{k}})}
\end{align}
where each $G_{i_{k}^{j}}, j=1,2,3,$ is assumed to be an unbiased estimator of $G(z),$ meaning that for $\xi_{i,j,k}(z) := G(z) - G_{i_{k}^{j}}(z),\; \mathbb{E}[\xi_{i,j,k}(z) |z] = 0.$
With these modifications, we may keep the update \eqref{c_update_1} the same in the stochastic setting.
First, we offer a lemma that clarifies the behavior of $\alpha_{k}$; our primary modification the original version, due to \cite{yoon_ryu}, is an updated bound for $\alpha_{0}$.
\begin{lemma} \label{funny_alpha_lemma}
    The sequence $\alpha_{k}$ \eqref{normal_alpha} starting at $\alpha_{0} \in (0, \eta), \; \eta := \min \{\frac{3}{4R\sqrt{\overline{K_{G}}}}, \frac{1}{\sqrt{2}R}\}$, monotonically decreases to a positive limit.
    In particular, when $\sqrt{\overline{K_{G}}} = 1,$ we recover \cref{normal_alpha_lemma}.
\end{lemma}

\begin{remark}
    Squeezing down the interval where $\alpha_{0}$ may start is a choice made to force the positivity of the term $(1 - \alpha_{k}^{2}R^{2} - \beta_{k})$ in \eqref{back_to_funny_alpha_lemma} for ease of analysis.
    With a different choice of $\beta_{k}$, one may wish to modify the upper bound $\eta$ by choosing the second term to be $\displaystyle \frac{\sqrt{1 - \hat{\beta}}}{R},$ where $\hat{\beta} := \sup \beta_{k}$ is not equal to $1.$
\end{remark}

\begin{proof}
    We assume $R = 1$ and $\sqrt{\overline{K_{G}}} = 1$ without loss of generality.
    We may rewrite \eqref{normal_alpha} as
    \begin{equation}
        \alpha_{k} - \alpha_{k+1} = \frac{\alpha_{k}^{3}}{(k+1)(k+3)(1 - \alpha_{k}^{2})}. \label{alpha_difference}
    \end{equation}
    Suppose that we have established that for some $N \geq 0, \; 0 < \alpha_{N} < \rho$ for some $\rho \in (0,1)$ that satisfies 
    \begin{equation}
        \eta := \frac{1}{2}\Big{(}\frac{1}{N + 1} + \frac{1}{N+2} \Big{)}\frac{\rho^{2}}{1-\rho^{2}} < 1. \label{eta_and_rho}
    \end{equation}
    \eqref{eta_and_rho} holds for all $N \geq 0$ if $\rho < \frac{3}{4}.$
    We now show that with \eqref{eta_and_rho},
    \begin{equation*}
        \alpha_{N} > \alpha_{N+1} > \cdots > \alpha_{N+k} > (1-\eta)\alpha_{N} \text{ for all }k > 0,
    \end{equation*}
    allowing us to obtain $\alpha_{k}$ as a monotonically decreasing sequence to some $\alpha$ such that $\alpha \geq (1-\eta)\alpha_{N}$.
    It suffices to prove $(1-\eta)\alpha_{N} < \alpha_{N+k} < \rho$ for all $k \geq 0,$ as \eqref{alpha_difference} indicates that $\{\alpha_{k}\}_{k=0}^{\infty}$ is decreasing.

    Use induction on $k$ to prove that $\alpha_{N+k} \in ((1-\eta)\alpha_{N},\rho)$.
    The case $k=0$ is trivial.
    Now suppose that $(1-\eta)\alpha_{N} < \alpha_{N+j} < \rho$ holds true for $j = 0, \ldots, k.$
    Then by \eqref{alpha_difference}, for each $0 \leq j \leq k$ we have
    \begin{align*}
        0 < \alpha_{N} - \alpha_{N+k+1} <& \sum_{j=0}^{k} \frac{1}{(N+j+1)(N+j+3)} \frac{\rho^{2}\alpha_{N}}{1-\rho^{2}} \\
        <& \frac{\rho^{2}\alpha_{N}}{1-\rho^{2}} \sum_{j=0}^{\infty} \frac{1}{(N+j+1)(N+j+3)} \\
        =& \frac{\rho^{2}\alpha_{N}}{1-\rho^{2}} \frac{1}{2} \Big{(} \frac{1}{N+1} \frac{1}{N+2}\Big{)} = \eta \alpha_{N},
    \end{align*}
    which gives $(1-\eta) \alpha_{N} < \alpha_{N+k+1} < \alpha_{N} < \rho,$ completing the induction.
\end{proof}
We need a careful Lyapunov analysis to handle the newly introduced stochasticity.
Rather than focusing on making a \textit{nonincreasing} Lyapunov functional, we aim to control how negative the differences between subsequent terms may be via variances.
The analysis here is inspired by the analogous stochastic Lyapunov lemma in \cite{lee2021fast}.

\begin{lemma}[Stochastic Lyapunov Functional, Moving Anchor EAG-V] \label{stochastic_descent_lemma}
    Consider the stochastic EAG-V with moving anchor \eqref{stoch eagv update1}, \eqref{stoch eagv update2}, \eqref{stoch eagv update3}, \eqref{gamma_update_2}, \eqref{c_update_1} along with conditions \ref{condition1}, \ref{condition2}, \ref{condition3}, and $\{i_{k}^{1}, i_{k}^{2}, i_{k}^{3}\}_{k=1}^{\infty}$ as previously described.
    Suppose we are given the sequences $\{A_{k}\}_{k=0}^{\infty}, \{B_{k}\}_{k=0}^{\infty}$ as described in \cref{LpnvLm_1} and the sequence $\{\alpha_{k}\}_{k=0}^{\infty}$ described in \cref{funny_alpha_lemma}. 
    Define the stochastic Lyapunov functional as
    \begin{align}
        V_{k} = A_{k}\|G(z^{k})\|^{2}+B_{k}\langle G(z^{k}), z^{k}-\bar{z}^{k}\rangle + c_k\|z^{*} - \bar{z}^{k}\|^{2}, \label{stoch_fxnal}
    \end{align}
    Then, \eqref{stoch_fxnal} satisfies the following:
    \begin{align*}
        \mathbb{E}[V_{k} - V_{k+1} | \bar{z}^{k},z^{k}] & \geq \\
         -2A_{k}\alpha_{k}R\cdot Var(z^{k})& - \frac{2A_{k}}{1-\beta_{k}}\alpha_{k}R\mathbb{E}[Var(z^{k+1/2})|\bar{z}^{k},z^{k}]
    \end{align*}
\end{lemma}

\begin{proof}
    With our Lyapunov functional \eqref{stoch_fxnal} in mind, we derive the following useful relations:
    \begin{align}
        z^{k} - z^{k+1} &= \beta_{k}(z^{k} - \bar{z}^{k}) + \alpha_{k}G_{i_{k}^{2}}(z^{k+1/2})\label{stoch_relation_1} \\
        z^{k+1/2}-z^{k+1} &= \alpha_{k}\big{(}G_{i_{k}^{2}}(z^{k+1/2}) - G_{i_{k}^{1}}(z^{k}) \big{)}\label{stoch_relation_2} \\ 
        \bar{z}^{k}-z^{k+1} &= (1-\beta_{k})(\bar{z}^{k} - z^{k}) + \alpha_{k}G_{i_{k}^{2}}(z^{k+1/2}) \label{stoch_relation_3} \\
        \bar{z}^{k} - \bar{z}^{k+1} &= -\Tilde{\gamma}_{k+1}G_{i_{k}^{3}}(z^{k+1}) \label{stoch_relation_4}.
    \end{align}
    \eqref{stoch_relation_1} is $z^{k}$ subtract \eqref{stoch eagv update1}, \eqref{stoch_relation_2} is \eqref{stoch eagv update2} subtract \eqref{stoch eagv update1}, \eqref{stoch_relation_3} is $\bar{z}^{k}$ subtract \eqref{stoch eagv update2}, and \eqref{stoch_relation_4} is \eqref{stoch eagv update3} rearranged. 
    As already evidenced, much of this proof will parallel the previous descending Lyapunov lemmas from the deterministic cases, but our end goal is to capture how negative the differences can be rather than force positivity.
    We introduce a nonnegative inner product to begin the process of simplifying:
    \begin{align*}
        V_{k}& - V_{k+1} \geq \\
        & A_{k}\|G(z^{k})\|^{2} + B_{k}\langle G(z^{k}), z^{k} - \bar{z}^{k}\rangle - A_{k+1}\|G(z^{k+1})\|^{2} \\
        &- B_{k+1}\langle G(z^{k+1}),z^{k+1} - \bar{z}^{k+1} \rangle - \frac{B_{k}}{\beta_{k}}\langle z^{k} - z^{k+1},G(z^{k}) - G(z^{k+1}) \rangle \\
        &+ c_{k}\|z^{*} - \bar{z}^{k}\|^{2} - c_{k+1}\|z^{*} - \bar{z}^{k+1}\|^{2}
    \end{align*}
    After some additional computation utilizing \eqref{stoch_relation_1} through \eqref{stoch_relation_4}, we obtain
    \begin{align*}
        V_{k}& - V_{k+1} \geq \\
        & \underbrace{A_{k}\|G(z^{k})\|^{2} - A_{k+1}\|G(z^{k+1})\|^{2} - \frac{\alpha_{k}B_{k}}{\beta_{k}}\langle G_{i_{k}^{2}}(z^{k+1/2}), G(z^{k}) - G(z^{k+1})\rangle}_{\text{I}} \\
        &\underbrace{+\alpha_{k}B_{k+1}\langle G(z^{k+1}),G_{i_{k}^{2}}(z^{k+1/2}) \rangle}_{\text{I}}
        \underbrace{+ c_{k}\|z^{*} - \bar{z}^{k}\|^{2} - c_{k+1}\|z^{*} - \bar{z}^{k+1}\|^{2}}_{\text{II}} \\
        &\underbrace{+B_{k+1}\Tilde{\gamma}_{k+1}\langle G_{i_{k}^{3}}(z^{k+1}), G(z^{k+1}) \rangle}_{\text{II}}. 
    \end{align*}
    From here, we will deal with I and II separately. 
    We will deal with II first.
    To begin, let's analyze the inner product contained within II under expectation:
    \begin{align}
        &\mathbb{E}[B_{k+1}\Tilde{\gamma}_{k+1}\langle G_{i_{k}^{3}}(z^{k+1}), G(z^{k+1}) \rangle|\bar{z}^{k}, z^{k}] \label{stoch_inner_prod_1} \\
        =&B_{k+1}\Tilde{\gamma}_{k+1}\mathbb{E}\big{[} \mathbb{E}[\langle G_{i_{k}^{3}}(z^{k+1}),G(z^{k+1})\rangle|z^{k+1},\bar{z}^{k},z^{k}]\big{|}\bar{z}^{k}, z^{k}\big{]} \label{stoch_inner_prod_2} \\
        =&B_{k+1}\Tilde{\gamma}_{k+1}\mathbb{E}[\|G(z^{k+1})\|^{2}|\bar{z}^{k},z^{k}] \label{stoch_inner_prod_3} \\
        \geq& B_{k+1}\Tilde{\gamma}_{k+1}\mathbb{E}\Big{[}\frac{\|G_{i_{k}^{3}}(z^{k+1})\|^{2}}{K_{G}(z^{k+1})}|\bar{z}^{k},z^{k}\Big{]} \label{stoch_inner_prod_4} \\
        =& B_{k+1}\mathbb{E}\Big{[}\frac{\|\bar{z}^{k} - \bar{z}^{k+1}\|^{2}}{\Tilde{\gamma}_{k+1}K_{G}(z^{k+1})}|\bar{z}^{k},z^{k}\Big{]}. \label{stoch_inner_prod_5}
    \end{align}
    From \eqref{stoch_inner_prod_1} to \eqref{stoch_inner_prod_2} to \eqref{stoch_inner_prod_3}, we apply the law of iterated expectation to get $G(z^{k+1})$.
    Knowing $z^{k+1}$, we recall that $G_{i_{k}^{3}}$ is an unbiased estimator of $G$ to get \eqref{stoch_inner_prod_3}.
    The inequality \eqref{stoch_inner_prod_4} results from \ref{condition2}, and \eqref{stoch_inner_prod_5} results from \eqref{stoch_relation_4}.

    Thus after taking expectation, II changes into the following:
    \begin{align}
        \mathbb{E}&[\text{II} |\bar{z}^{k},z^{k}] \\
        \geq &\mathbb{E}[c_{k}\|z^{*} - \bar{z}^{k}\|^{2} - c_{k+1}\|z^{*} - \bar{z}^{k+1}\|^{2} + \frac{B_{k+1}}{\Tilde{\gamma}_{k+1}K_{G}}\|\bar{z}^{k}-\bar{z}^{k+1}\|^{2}|\bar{z}^{k},z^{k}] \\
        \geq &\mathbb{E}[c_{k}\|z^{*} - \bar{z}^{k}\|^{2} - c_{k+1}\big{(}(1 + \delta_{k})\|z^{*} - \bar{z}^{k}\|^{2} + (1+\frac{1}{\delta_{k}})\|\bar{z}^{k} - \bar{z}^{k+1}\|^{2}\big{)} \notag \\
        + &\frac{B_{k+1}}{\Tilde{\gamma}_{k+1}K_{G}(z^{k+1})}\|\bar{z}^{k}-\bar{z}^{k+1}\|^{2} |\bar{z}^{k},z^{k}] \label{cauchysch}
    \end{align}
    where \eqref{cauchysch} is an application of Cauchy-Schwartz to $\|z^{*} - \bar{z}^{k+1}\|$.
    Because \newline $\Tilde{\gamma}_{k+1} = \frac{1}{\overline{K_{G}}}\frac{B_{k+1}}{c_{k+1}(1+\frac{1}{\delta_{k}})}$, we find
    \[
        \mathbb{E}[\text{II}|\bar{z}^{k},z^{k}] \geq 0,
    \]
    and now we are left with I:
    \begin{align*}
        \mathbb{E}[V_{k} &- V_{k+1}|\bar{z}^{k},z^{k}] \geq \mathbb{E}[I|\bar{z}^{k},z^{k}]\\
        =\mathbb{E}&\big{[} A_{k}\|G(z^{k})\|^{2} - A_{k+1}\|G(z^{k+1})\|^{2} - \frac{\alpha_{k}B_{k}}{\beta_{k}}\langle G_{i_{k}^{2}}(z^{k+1/2}),G(z^{k}) - G(z^{k+1}) \rangle \\
        & + \alpha_{k}B_{k+1}\langle G(z^{k+1}),G_{i_{k}^{2}}(z^{k+1/2}) \rangle \big{|} \bar{z}^{k}, z^{k} \big{]}.
    \end{align*}
    First, we note that 
    \begin{align}
        \|G(z^{k+1/2}) - G(z^{k+1})\|^{2} &\leq R^{2} \|z^{k+1/2} - z^{k+1}\|^{2} \notag \\
        &= R^{2} \alpha_{k}^2\|G_{i_{k}^{2}}(z^{k+1/2}) - G_{i_{k}^{1}}(z^{k})\|^{2} \\
        & \Leftrightarrow \notag \\
        -A_{k} \|G_{i_{k}^{2}}(z^{k+1/2}) - G_{i_{k}^{1}}(z^{k})\|^{2} + &\frac{A_{k}}{R^{2}\alpha_{k}^{2}}\|G(z^{k+1/2}) - G(z^{k+1})\|^{2} \leq 0 \label{R_smooth_stoch}
    \end{align}
    by $R-$smoothness, so that
    \begin{align}
        \mathbb{E}[\text{I}|&\bar{z}^{k},z^{k}] \geq \notag \\
        & \mathbb{E}\big{[} A_{k}\|G(z^{k})\|^{2} - A_{k+1}\|G(z^{k+1})\|^{2} - \frac{\alpha_{k}B_{k}}{\beta_{k}}\langle G_{i_{k}^{2}}(z^{k+1/2}),G(z^{k}) - G(z^{k+1}) \rangle \notag \\
        & + \alpha_{k}B_{k+1}\langle G(z^{k+1}),G_{i_{k}^{2}}(z^{k+1/2}) \rangle -A_{k} \|G_{i_{k}^{2}}(z^{k+1/2}) - G_{i_{k}^{1}}(z^{k})\|^{2} \notag \\
        & + \frac{A_{k}}{R^{2}\alpha_{k}^{2}}\|G(z^{k+1/2}) - G(z^{k+1})\|^{2} \big{|} \bar{z}^{k}, z^{k}\big{]} \label{I_negative_term1} \\
        = & \mathbb{E}\big{[} A_{k}\|G(z^{k})\|^{2} - A_{k+1}\|G(z^{k+1})\|^{2} - \frac{\alpha_{k}B_{k}}{\beta_{k}}\langle G_{i_{k}^{2}}(z^{k+1/2}),G(z^{k}) - G(z^{k+1}) \rangle \notag \\
        & + \alpha_{k}B_{k+1}\langle G(z^{k+1}),G_{i_{k}^{2}}(z^{k+1/2}) \rangle \notag \\
        & -A_{k} \big{(} \|G_{i_{k}^{2}}(z^{k+1/2})\|^{2} - 2\langle G_{i_{k}^{2}}(z^{k+1/2}),G_{i_{k}^{1}}(z^{k}) \rangle + \|G_{i_{k}^{1}}(z^{k})\|^{2} \big{)} \notag \\
        & + \frac{A_{k}}{R^{2}\alpha_{k}^{2}} \big{(}\|G(z^{k+1/2})\|^{2} - 2\langle G(z^{k+1/2}), G(z^{k+1})\rangle +\|G(z^{k+1})\|^{2}\big{)} \big{|} \bar{z}^{k}, z^{k} \big{]} \label{I_negative_term2}
    \end{align}
    To be clear, \eqref{I_negative_term1} is I subtract \eqref{R_smooth_stoch}, and \eqref{I_negative_term2} is \eqref{I_negative_term1} with the terms introduced from \eqref{R_smooth_stoch} expanded.
    We rearrange \eqref{I_negative_term2} below to visualize cancellations and groupings of terms:
    \begin{align}
        \mathbb{E}\big{[} &A_{k}\|G(z^{k})\|^{2} -A_{k}\|G_{i_{k}^{1}}(z^{k})\|^{2} -A_{k+1}\|G(z^{k+1})\|^{2} + \frac{A_{k}}{R^{2}\alpha_{k}^{2}}\|G(z^{k+1})\|^{2} \notag \\
        +&(\frac{\alpha_{k}B_{k}}{\beta_{k}} + \alpha_{k}B_{k+1})\langle G_{i_{k}^{2}}(z^{k+1/2}),G(z^{k+1})\rangle - \frac{\alpha_{k}B_{k}}{\beta_{k}}\langle G_{i_{k}^{2}}(z^{k+1/2}),G(z^{k}) \rangle\notag \\
        -& A_{k}\|G_{i_{k}^{2}}(z^{k+1/2}))\|^{2} + \frac{A_{k}}{R^{2}\alpha_{k}^{2}}\|G(z^{k+1/2})\|^{2} + 2A_{k}\langle G_{i_{k}^{2}}(z^{k+1/2}),G_{i_{k}^{1}}(z^{k})\rangle \notag \\
        -&\frac{2A_{k}}{R^{2}\alpha_{k}^{2}}\langle G(z^{k+1/2}),G(z^{k+1}) \rangle\big{|} \bar{z}^{k}, z^{k}\big{]} \label{need_stoch_cancels} \\
        \geq \mathbb{E}\big{[} & (\frac{A_{k}}{R^{2}\alpha_{k}^{2}} - A_{k}K_{G}(z^{k+1/2}))\|G(z^{k+1/2})\|^{2} + (\frac{A_{k}}{R^{2}\alpha_{k}^{2}} - A_{k+1})\|G(z^{k+1})\|^{2} \notag \\
        +&2A_{k}\langle G_{i_{k}^{2}}(z^{k+1/2}), G_{i_{k}^{1}}(z^{k})\rangle - \frac{\alpha_{k}B_{k}}{\beta_{k}} \langle G_{i_{k}^{2}}(z^{k+1/2}), G(z^{k})\rangle \notag \\
        +&(\frac{\alpha_{k}B_{k}}{\beta_{k}} + \alpha_{k}B_{k+1})\langle G_{i_{k}^{2}}(z^{k+1/2}), G(z^{k+1})\rangle - \frac{2A_{k}}{R^{2}\alpha_{k}^{2}}\langle G(z^{k+1/2}), G(z^{k+1})\rangle \big{|} \bar{z}^{k}, z^{k}\big{]}, \label{stoch_some_cancels}
    \end{align}
    where the first two terms of \eqref{need_stoch_cancels} cancel by the law of iterated expectation applied to $\|G_{i_{k}^{1}}(z^{k})\|^{2}$ and the first term in \eqref{stoch_some_cancels} comes from applying \ref{condition2} to $- A_{k}\|G_{i_{k}^{2}}(z^{k+1/2}))\|^{2} + \frac{A_{k}}{R^{2}\alpha_{k}^{2}}\|G(z^{k+1/2})\|^{2}$.
    Now, we may apply the law of iterated expectation to modify the two terms in the second line of \eqref{stoch_some_cancels}:
    \begin{align}
        \mathbb{E}\big{[}& 2A_{k}\langle G_{i_{k}^{2}}(z^{k+1/2}), G_{i_{k}^{1}}(z^{k})\rangle - \frac{\alpha_{k}B_{k}}{\beta_{k}} \langle G_{i_{k}^{2}}(z^{k+1/2}), G(z^{k})\rangle \big{|} \bar{z}^{k},z^{k} \big{]} \notag \\
        =\mathbb{E}\big{[}& \mathbb{E}[ 2A_{k}\langle G_{i_{k}^{2}}(z^{k+1/2}), G_{i_{k}^{1}}(z^{k})\rangle - \frac{\alpha_{k}B_{k}}{\beta_{k}} \langle G_{i_{k}^{2}}(z^{k+1/2}), G(z^{k})\rangle |\bar{z}^{k}, z^{k}, i_{k}^{1}] \big{|} \bar{z}^{k}, z^{k}\big{]} \notag \\
        =\mathbb{E}\big{[}& 2A_{k}\langle G(z^{k+1/2}), G_{i_{k}^{1}}(z^{k})\rangle - \frac{\alpha_{k}B_{k}}{\beta_{k}} \langle G(z^{k+1/2}), G(z^{k})\rangle \big{|} \bar{z}^{k},z^{k} \big{]}.\label{application_of_law}
    \end{align}
    With observation \eqref{application_of_law} under our belts, we continue by introducing some terms at the tail end of an updated \eqref{stoch_some_cancels}:
    \begin{align}
        \mathbb{E}\big{[} & (\frac{A_{k}}{R^{2}\alpha_{k}^{2}} - A_{k}K_{G}(z^{k+1/2}))\|G(z^{k+1/2})\|^{2} + (\frac{A_{k}}{R^{2}\alpha_{k}^{2}} - A_{k+1})\|G(z^{k+1})\|^{2} \notag \\
        +&2A_{k}\langle G(z^{k+1/2}), G_{i_{k}^{1}}(z^{k})\rangle - \frac{\alpha_{k}B_{k}}{\beta_{k}} \langle G(z^{k+1/2}), G(z^{k})\rangle \notag \\
        +&(\frac{\alpha_{k}B_{k}}{\beta_{k}} + \alpha_{k}B_{k+1})\langle G_{i_{k}^{2}}(z^{k+1/2}), G(z^{k+1})\rangle - \frac{2A_{k}}{R^{2}\alpha_{k}^{2}}\langle G(z^{k+1/2}), G(z^{k+1})\rangle \notag \\
        +&2A_{k}\langle G(z^{k+1/2}),G(z^{k})\rangle -2A_{k}\langle G(z^{k+1/2}),G(z^{k})\rangle \notag \\
        +&(\frac{\alpha_{k}B_{k}}{\beta_{k}} + \alpha_{k}B_{k+1})\langle G(z^{k+1}),G(z^{k+1/2}) \rangle \notag \\
        -&(\frac{\alpha_{k}B_{k}}{\beta_{k}} + \alpha_{k}B_{k+1})\langle G(z^{k+1}),G(z^{k+1/2}) \rangle\big{|} \bar{z}^{k}, z^{k}\big{]} \notag \\
        \geq \mathbb{E}\big{[}& (\frac{A_{k}}{R^{2}\alpha_{k}^{2}} - A_{k}K_{G}(z^{k+1/2}))\|G(z^{k+1/2})\|^{2} + (\frac{A_{k}}{R^{2}\alpha_{k}^{2}} - A_{k+1})\|G(z^{k+1})\|^{2} 
        \notag\\
        +&(2A_{k} - \frac{\alpha_{k}B_{k}}{\beta_{k}})\langle G(z^{k+1/2}),G(z^{k}) \rangle \label{almost_to_young_inequality2} \\ 
        +&(\frac{\alpha_{k}B_{k}}{\beta_{k}} + \alpha_{k}B_{k+1} - \frac{2A_{k}}{R^{2}\alpha_{k}^{2}})\langle G(z^{k+1}),G(z^{k+1/2}) \rangle 
        \notag\\
        +& 2A_{k} \langle G(z^{k+1/2}),G_{i_{1}^{k}}(z^{k}) - G(z^{k}) \rangle \label{almost_to_young_inequality4} \\
        +&(\frac{\alpha_{k}B_{k}}{\beta_{k}} +\alpha_{k}B_{k+1}) \langle G(z^{k+1}),G_{i_{k}^{2}}(z^{k+1/2}) - G(z^{k+1/2}) \rangle \big{|} \bar{z}^{k}, z^{k} \big{]} \label{almost_to_young_inequality5}
    \end{align}
    Let us momentarily ignore the latter two terms \eqref{almost_to_young_inequality4}, \eqref{almost_to_young_inequality5}.
    Note \eqref{almost_to_young_inequality2} is zero by the definition of $A_{k}$, and for the other coefficients,
    \begin{align}
        \frac{A_{k}}{R^{2}\alpha_{k}^{2}} - A_{k}K_{G}(z^{k+1/2}) & = \frac{A_{k}(1 - \alpha_{k}^{2}R^{2}K_{G}(z^{k+1/2}))}{\alpha_{k}^{2}R^{2}} \label{gzk_plus_half_coeff} \\
        A_{k+1} & = \frac{A_{k}(1 - \alpha_{k}^{2}R^{2} - \beta_{k}^{2})}{(1 - \alpha_{k}^{2}R^{2})(1 - \beta_{k})^{2}} \notag \\
        \implies \frac{A_{k}}{\alpha_{k}^{2}R^{2}} - A_{k+1} & = \frac{A_{k}(1 - \alpha_{k}^{2}R^{2} - \beta_{k})^{2}}{\alpha_{k}^{2}R^{2}(1 - \alpha_{k}^{2}R^{2})(1 - \beta_{k})^{2}} \label{gzk_plus_one_coeff} \\
        \alpha_{k}B_{k+1} + \frac{\alpha_{k}B_{k}}{\beta_{k}} -\frac{2A_{k}}{\alpha_{k}^{2}R^{2}} & = \frac{2A_{k}}{1 - \beta_{k}} -\frac{2A_{k}}{\alpha_{k}^{2}R^{2}} \notag \\
        & = - \frac{2A_{k}(1 - \alpha_{k}^{2}R^{2} - \beta_{k})}{\alpha_{k}^{2}R^{2}(1 - \beta_{k})} \label{cross_term_coeff}
    \end{align}
    which, if we continue ignoring \eqref{almost_to_young_inequality4} and \eqref{almost_to_young_inequality5} while substituting in \eqref{gzk_plus_half_coeff}, \eqref{gzk_plus_one_coeff}, and \eqref{cross_term_coeff}, yields
    
    \begin{align}
        V_{k}& - V_{k+1} \geq \notag \\
        \mathbb{E}&[\text{I}|\bar{z}^{k},z^{k}] \geq \notag \\
        \mathbb{E}&\big{[}\big{(} \frac{A_{k}(1 - \alpha_{k}^{2}R^{2}K_{G}(z^{k+1/2}))}{\alpha_{k}^{2}R^{2}}\big{)}\|G(z^{k + 1/2})\|^{2} \notag \\
        &+ \big{(} \frac{A_{k}(1 - \alpha_{k}^{2}R^{2} - \beta_{k})^{2}}{\alpha_{k}^{2}R^{2}(1 - \alpha_{k}^{2}R^{2})(1 - \beta_{k})^{2}}\big{)}\|G(z^{k+1})\|^{2} \label{two_squared_terms}\\
        -&\frac{2A_{k}(1 - \alpha_{k}^{2}R^{2} -\beta_{k})}{\alpha_{k}^{2}R^{2}(1 - \beta_{k})}\langle G(z^{k+1}),G(z^{k+1/2}) \rangle \big{|} \bar{z}^{k}, z^{k} \big{]}. \label{full_cross_term}
    \end{align}
    
    Our aim is to complete the square via Young's inequality to demonstrate the nonnegativity of these terms.
    A slight complicating factor exists in the extra $K_{G}(z^{k+1/2})$ in a coefficient within \eqref{two_squared_terms}, which we deal with in the following way.
    \begin{align}
        \big{(}& \frac{A_{k}(1 - \alpha_{k}^{2}R^{2}K_{G}(z^{k+1/2}))}{\alpha_{k}^{2}R^{2}}\big{)}\|G(z^{k + 1/2})\|^{2} + \big{(} \frac{A_{k}(1 - \alpha_{k}^{2}R^{2} - \beta_{k})^{2}}{\alpha_{k}^{2}R^{2}(1 - \alpha_{k}^{2}R^{2})(1 - \beta_{k})^{2}}\big{)}\|G(z^{k+1})\|^{2} \notag \\
        -&\frac{2A_{k}(1 - \alpha_{k}^{2}R^{2} -\beta_{k})}{\alpha_{k}^{2}R^{2}(1 - \beta_{k})}\langle G(z^{k+1}),G(z^{k+1/2}) \rangle \notag \\
        =&\big{(} 1 - \alpha_{k}^{2}R^{2}K_{G}(z^{k+1/2})\big{)}\|G(z^{k + 1/2})\|^{2} + \frac{(1 - \alpha_{k}^{2}R^{2} - \beta_{k})^{2}}{(1 - \alpha_{k}^{2}R^{2})(1 - \beta_{k})^{2}}\|G(z^{k+1})\|^{2} \notag \\
        -&\frac{2(1 - \alpha_{k}^{2}R^{2} -\beta_{k})}{(1 - \beta_{k})}\langle G(z^{k+1}),G(z^{k+1/2}) \rangle \label{lose_the_coefficient} \\
        =& \|\sqrt{1 - \alpha_{k}^{2}R^{2}K_{G}(z^{k+1/2})}G(z^{k + 1/2})\|^{2} + \frac{(1 - \alpha_{k}^{2}R^{2} - \beta_{k})^{2}}{(1 - \alpha_{k}^{2}R^{2})}\Big{\|}\frac{G(z^{k+1})}{(1 - \beta_{k})}\Big{\|}^{2} \notag \\
        -&2\Big{\langle} \frac{(1 - \alpha_{k}^{2}R^{2} - \beta_{k})}{(1 - \beta_{k})\sqrt{1 - \alpha_{k}^{2}R^{2}K_{G}(z^{k+1/2})}}G(z^{k+1}),\sqrt{1 - \alpha_{k}^{2}R^{2}K_{G}(z^{k+1/2})}G(z^{k+1/2}) \Big{\rangle} \notag \\
        +&(1 - \alpha_{k}^{2}R^{2} - \beta_{k})^{2}\Big{\|}\frac{G(z^{k+1})}{(1 - \beta_{k})\sqrt{1 - \alpha_{k}^{2}R^{2}K_{G}(z^{k+1/2})}}\Big{\|}^{2} \notag \\
        -&(1 - \alpha_{k}^{2}R^{2} - \beta_{k})^{2}\Big{\|}\frac{G(z^{k+1})}{(1 - \beta_{k})\sqrt{1 - \alpha_{k}^{2}R^{2}K_{G}(z^{k+1/2})}}\Big{\|}^{2} \notag \\
    \end{align}
\begin{align}
        \geq & \frac{(1 - \alpha_{k}^{2}R^{2} - \beta_{k})^{2}}{(1 - \alpha_{k}^{2}R^{2})}\Big{\|}\frac{G(z^{k+1})}{(1 - \beta_{k})}\Big{\|}^{2} - \frac{(1 - \alpha_{k}^{2}R^{2} - \beta_{k})^{2}}{(1 - \alpha_{k}^{2}R^{2}K_{G}(z^{k+1/2}))}\Big{\|} \frac{G(z^{k+1})}{(1 - \beta_{k})} \Big{\|}^{2} \label{app_of_YE} \\
        =& (1 - \alpha_{k}^{2}R^{2} - \beta_{k})\Big{\|} \frac{G(z^{k+1})}{(1 - \beta_{k})} \Big{\|}^{2}\bigg{(} \frac{(1 - \alpha_{k}^{2}R^{2} - \beta_{k})}{(1 - \alpha_{k}^{2}R^{2})} - \frac{(1 - \alpha_{k}^{2}R^{2} - \beta_{k})}{(1 - \alpha_{k}^{2}R^{2}K_{G}(z^{k+1/2}))} \bigg{)}\label{back_to_funny_alpha_lemma} \\
        =& (1 - \alpha_{k}^{2}R^{2} - \beta_{k})\Big{\|} \frac{G(z^{k+1})}{(1 - \beta_{k})} \Big{\|}^{2} \bigg{(} 1 - \frac{(1 - \alpha_{k}^{2}R^{2})}{(1 - \alpha_{k}^{2}R^{2}K_{G}(z^{k+1/2}))} \notag\\
        &+ \beta_{k} \Big{(}\frac{1}{(1 - \alpha_{k}^{2}R^{2}K_{G}(z^{k+1/2}))} - \frac{1}{(1 - \alpha_{k}^{2}R^{2})} \Big{)}\bigg{)} \notag \\
        =& (1 - \alpha_{k}^{2}R^{2} - \beta_{k})\Big{\|} \frac{G(z^{k+1})}{(1 - \beta_{k})} \Big{\|}^{2} \bigg{(} \frac{\alpha_{k}^{2}R^{2}(1 - K_{G}(z^{k+1/2}))}{(1 - \alpha_{k}^{2}R^{2}K_{G}(z^{k+1/2}))} \notag \\
        &+ \frac{\alpha_{k}^{2}R^{2}\beta_{k}(K_{G}(z^{k+1/2}) - 1)}{(1 - \alpha_{k}^{2}R^{2}K_{G}(z^{k+1/2}))(1 - \alpha_{k}^{2}R^{2})} \bigg{)} \notag \\
        >& (1 - \alpha_{k}^{2}R^{2} - \beta_{k})\Big{\|} \frac{G(z^{k+1})}{(1 - \beta_{k})} \Big{\|}^{2} \bigg{(} \frac{\alpha_{k}^{2}R^{2}\beta_{k}(1 - K_{G}(z^{k+1/2}))}{(1 - \alpha_{k}^{2}R^{2}K_{G}(z^{k+1/2}))} \notag \\
        &+ \frac{\alpha_{k}^{2}R^{2}\beta_{k}(K_{G}(z^{k+1/2}) - 1)}{(1 - \alpha_{k}^{2}R^{2}K_{G}(z^{k+1/2}))(1 - \alpha_{k}^{2}R^{2})} \bigg{)} \label{second_to_last_ineq} \\
        >& (1 - \alpha_{k}^{2}R^{2} - \beta_{k})\Big{\|} \frac{G(z^{k+1})}{(1 - \beta_{k})} \Big{\|}^{2} \bigg{(} \frac{\alpha_{k}^{2}R^{2}\beta_{k}(1 - K_{G}(z^{k+1/2}))}{(1 - \alpha_{k}^{2}R^{2}K_{G}(z^{k+1/2}))} \notag \\
        &+ \frac{\alpha_{k}^{2}R^{2}\beta_{k}(K_{G}(z^{k+1/2}) - 1)}{(1 - \alpha_{k}^{2}R^{2}K_{G}(z^{k+1/2}))} \bigg{)} \label{final_ineq} \\
        =& 0, \notag 
    \end{align}
   and we note \eqref{second_to_last_ineq} and \eqref{final_ineq} are due to $0 < \beta_{k} < 1$ and $\displaystyle 1 < \frac{1}{(1 - \alpha_{k}^{2}R^{2})},$ respectively.
    For clarity, the term $(1 - \alpha_{k}^{2}R^{2} - \beta_{k})$ is positive because of the starting point of $\alpha_{0}$ in \cref{funny_alpha_lemma}, so there are no issues with bringing this factor to the front of the expression for our analysis.
    Bringing back our other terms, this demonstrates that
    \begin{align}
        V_{k}& - V_{k+1} \geq \notag \\
        \mathbb{E}&[\text{I}|\bar{z}^{k},z^{k}] \geq \notag \\
        +& 2A_{k} \langle G(z^{k+1/2}),G_{i_{1}^{k}}(z^{k}) - G(z^{k}) \rangle \label{finally_1} \\
        +&(\frac{\alpha_{k}B_{k}}{\beta_{k}} +\alpha_{k}B_{k+1}) \langle G(z^{k+1}),G_{i_{k}^{2}}(z^{k+1/2}) - G(z^{k+1/2}) \rangle \big{|} \bar{z}^{k}, z^{k} \big{]} \label{finally_2} \\
        =+& 2A_{k} \langle G(z^{k+1/2}),G_{i_{1}^{k}}(z^{k}) - G(z^{k}) \rangle \notag \\
        +&\frac{2A_{k}}{1 - \beta_{k}} \langle G(z^{k+1}),G_{i_{k}^{2}}(z^{k+1/2}) - G(z^{k+1/2}) \rangle \big{|} \bar{z}^{k}, z^{k} \big{]} \notag 
    \end{align}
    For \eqref{finally_1}, we note that
    \begin{equation}
        0 = \langle -G\big{(}z^{k} + \beta_{k}(\bar{z}^{k} - z^{k}) - \alpha_{k}G(z^{k})\big{)}, \mathbb{E}[G_{i_{k}^{1}}(z^{k}) - G(z^{k})|\bar{z}^{k},z^{k}]\rangle, \label{first_zero_thing}
    \end{equation}
    which allows us to compute
    \begin{align}
        \big{|}&\mathbb{E}[ \langle G(z^{k+1/2}), G_{i_{k}^{1}}(z^{k}) - G(z^{k})\rangle |\bar{z}^{k},z^{k}]\big{|} \notag \\
        =&\big{|}\mathbb{E}[ \langle G(z^{k+1/2}) - G\big{(}z^{k} + \beta_{k}(\bar{z}^{k} - z^{k}) - \alpha_{k}G(z^{k})\big{)}, G_{i_{k}^{1}}(z^{k}) - G(z^{k})\rangle |\bar{z}^{k},z^{k}]\big{|} \notag \\
        \leq&\mathbb{E}\big{[} \big{\|}G(z^{k+1/2}) - G\big{(}z^{k} + \beta_{k}(\bar{z}^{k} - z^{k}) - \alpha_{k}G(z^{k})\big{)}\big{\|} \cdot \big{\|}G_{i_{k}^{1}}(z^{k}) - G(z^{k})\big{\|} | |\bar{z}^{k},z^{k}\big{]} \notag \\
        \leq&\mathbb{E}\big{[} R\big{\|}z^{k+1/2} - \big{(}z^{k} + \beta_{k}(\bar{z}^{k} - z^{k}) - \alpha_{k}G(z^{k})\big{)}\big{\|} \cdot \big{\|}G_{i_{k}^{1}}(z^{k}) - G(z^{k})\big{\|} | |\bar{z}^{k},z^{k}\big{]} \notag \\
        =&\mathbb{E}[R\alpha_{k}\|G_{i_{k}^{1}}(z^{k}) - G(z^{k})\|^{2}|\bar{z}^{k},z^{k}] \notag \\
        =&R\alpha_{k} Var(z^{k}), \label{variance_term_1}
    \end{align}
    via an application of Cauchy-Schwartz, the Lipschitz property, and the definition of the algorithm, where $Var(z) := \mathbb{E}[\langle G_{i}(z) - G(z), G_{i}(z) - G(z)\rangle | z]$ is the variance of the difference.
    Similarly, for \eqref{finally_2} we obtain
    \begin{align}
        \big{|}&\mathbb{E}[ \langle G(z^{k+1}), G_{i_{k}^{2}}(z^{k+1/2}) - G(z^{k+1/2})\rangle |\bar{z}^{k},z^{k}]\big{|} \notag \\
        =&\big{|}\mathbb{E}[ \langle G(z^{k+1}) - G\big{(}z^{k} + \beta_{k}(\bar{z}^{k} - z^{k}) - \alpha_{k}G(z^{k+1/2})\big{)}, G_{i_{k}^{2}}(z^{k+1/2}) - G(z^{k+1/2})\rangle |\bar{z}^{k},z^{k}]\big{|} \notag \\
        \leq&\mathbb{E}\big{[} \big{\|}G(z^{k+1}) - G\big{(}z^{k} + \beta_{k}(\bar{z}^{k} - z^{k}) - \alpha_{k}G(z^{k+1/2})\big{)}\big{\|} \cdot \big{\|}G_{i_{k}^{2}}(z^{k+1/2}) - G(z^{k+1/2})\big{\|}  |\bar{z}^{k},z^{k}\big{]} \notag \\
        \leq&\mathbb{E}\big{[} R\big{\|}z^{k+1} - \big{(}z^{k} + \beta_{k}(\bar{z}^{k} - z^{k}) - \alpha_{k}G(z^{k+1/2})\big{)}\big{\|} \cdot \big{\|}G_{i_{k}^{2}}(z^{k+1/2}) - G(z^{k+1/2})\big{\|}  |\bar{z}^{k},z^{k}\big{]} \notag \\
        =&\mathbb{E}\Big{[} \mathbb{E} \big{[} R\big{\|}z^{k+1} - \big{(}z^{k} + \beta_{k}(\bar{z}^{k} - z^{k}) - \alpha_{k}G(z^{k+1/2})\big{)}\big{\|} \notag \\
        & \cdot \big{\|}G_{i_{k}^{2}}(z^{k+1/2}) - G(z^{k+1/2}) \big{\|} | z^{k+1/2},\bar{z}^{k}, z^{k} \big{]} |\bar{z}^{k},z^{k}\Big{]} \notag \\
        =&\mathbb{E}[R\alpha_{k}\|G_{i_{k}^{2}}(z^{k+1/2}) - G(z^{k+1/2})\|^{2}|z^{k+1/2},\bar{z}^{k},z^{k}] \notag \\
        =&R\alpha_{k}\mathbb{E}[Var(z^{k+1/2})|\bar{z}^{k},z^{k}] \label{variance_term_2}
    \end{align}
    by the law of iterated expectation and reasoning similar to that of \eqref{finally_1}.
    Equipped with \eqref{variance_term_1} and \eqref{variance_term_2}, we get
    \begin{equation}
        V_{k} - V_{k+1} \geq -2A_{k}\alpha_{k}RVar(z^{k}) - \frac{2A_{k}}{1 - \beta_{k}}\alpha_{k}R\mathbb{E}[Var(z^{k+1/2})|\bar{z}^{k},z^{k}] \label{difference_bound}
    \end{equation}
    and the lemma is proved.
\end{proof}

To proceed towards convergence, we need the following.

\begin{definition}
    The \textit{filtration}
    \[\mathcal{F}^{k} = \sigma(z^{0}, \bar{z}^{0}, z^{1}, \bar{z}^{1}, \ldots, z^{k}, \bar{z}^{k}, i_{1}^{0}, i_{2}^{0}, i_{3}^{0}, \ldots, i_{1}^{k}, i_{2}^{k}, i_{3}^{k}) \]
    represents the history of of iterates, anchors, and choices of component $i$ up through the current step $k$.
\end{definition}

\begin{theorem}[Supermartingale Convergence Theorem \cite{combettes2015stochastic}]\label{smct}
    Let $P^{k}, J^{k},$ and $W^{k}$ be positive sequences adapted to $\mathcal{F}^{k},$ and suppose $W^{k}$ is summable with probability $1.$
    If
    \[\mathbb{E}[P^{k+1} | \mathcal{F}^{k}] + J^{k} \leq P^{k} + W^{k},\]
    then with probability $1$, $P^{k}$ converges to a $[0,\infty)-$valued random variable and $\sum_{j=1}^{\infty} J^{k} < \infty$.
\end{theorem}

Now, we may apply \cref{summability_condition} to satisfy the conditions of \cref{smct}.

\begin{lemma}[Summability of Variances] \label{summability_lemma}
    Consider the stochastic Lyapunov functional $V_{k}$ \eqref{stoch_fxnal} discussed in \cref{stochastic_descent_lemma} along with conditions \eqref{condition1}, \eqref{condition2}, \eqref{condition3}, the choice $\beta_{k} = \frac{1}{k+2}$ made in \cref{Lpnv_RMK}, and \cref{summability_condition}.
    Given \eqref{difference_bound}, the extraneous sequence of terms
    \[
        2A_{k}\alpha_{k}R\Big{(} Var(z^{k}) + \frac{\mathbb{E}[Var(z^{k+1/2})|\bar{z}^{k},z^{k}]}{1 - \beta_{k}} \Big{)}
    \]
    is summable.
\end{lemma}

\begin{proof}
    Because of \eqref{var_bound} in \eqref{condition2} and \cref{summability_condition}, it is sufficient to demonstrate that
    \begin{equation}
        \sum_{k=0}^{\infty} 2A_{k}\alpha_{k}R\Big{(}(C_{G}(z^{k}) - 1)\|G(z^{k})\|^{2} + \frac{1}{1-\beta_{k}}\mathbb{E}[\big{(}\overline{C_{G}}(z^{k+1/2})-1\big{)}\|G(z^{k+1/2})\|^{2}|\bar{z}^{k},z^{k}]\Big{)}
    \end{equation}
    is finite.
    By construction, $\alpha_{k}$ is a nonnegative term bounded above by the choice $\alpha_{0}$, so the following bound for each summand exists if we substitute in the definition of $A_{k}$ made in \cref{Lpnv_RMK}:
    \begin{align}
        &\alpha_{0}^{2}R(k+1)(k+2)\Big{(}(C_{G}(z^{k}) - 1)\|G(z^{k})\|^{2} \notag \\
        & + \frac{1}{1-\beta_{k}}\mathbb{E}[\big{(}\overline{C_{G}}(z^{k+1/2})-1\big{)}\|G(z^{k+1/2})\|^{2}|\bar{z}^{k},z^{k}]\Big{)} \notag \\
        \leq&\alpha_{0}^{2}R(k+1)(k+2)\Big{(}\frac{C_{1}}{(k+1)^{4}} + \frac{C_{2}}{(k+1)^{4}(1 - \beta_{k})}\Big{)} \label{we_love_catching_bounds1} \\
        =&\frac{\alpha_{0}^{2}R(k+1)(k+2)}{1-\beta_{k}}\Big{(}\frac{C_{1}(1 - \beta_{k}) + C_{2}}{(k+1)^{4}}\Big{)} \label{we_love_catching_bounds2} \\
        <&2\alpha_{0}^{2}R(k+1)(k+2)\Big{(}\frac{C_{1} + C_{2}}{(k+1)^{4}}\Big{)} \label{we_love_catching_bounds3}
    \end{align}
    where \eqref{we_love_catching_bounds1} results from \cref{summability_condition}, \eqref{we_love_catching_bounds2} results from rationalizing the denominator with $(1-\beta_{k})$, and the last inequality \eqref{we_love_catching_bounds3} is a result of the facts that $1 - \beta_{k} < 1, \frac{1}{1 - \beta_{k}} < 2$.
    For the summation, these facts result in
    \begin{align*}
        &\sum_{k=0}^{\infty} 2A_{k}\alpha_{k}R\Big{(}(C_{G}(z^{k}) - 1)\|G(z^{k})\|^{2} + \frac{1}{1-\beta_{k}}\mathbb{E}[\big{(}\overline{C_{G}}(z^{k+1/2})-1\big{)}\|G(z^{k+1/2})\|^{2}|\bar{z}^{k},z^{k}]\Big{)} \\
        <& 2\alpha_{0}^{2}R(C_{1} + C_{2})\sum_{k=0}^{\infty} \frac{k^{2} + 3k + 2}{(k+1)^{4}}\\
        =& 2\alpha_{0}^{2}R(C_{1} + C_{2})\sum_{k=0}^{\infty} \frac{k^{2}}{k^{4}} + \text{ other terms more summable than } \frac{1}{k^{2}} < \infty.
    \end{align*}
\end{proof}

\begin{theorem}[Stochastic Lyapunov Functional Convergence]
    Consider the stochastic Lyapunov functional $V_{k}$ \eqref{stoch_fxnal} along with conditions \eqref{condition1}, \eqref{condition2}, \eqref{condition3}, the choice $\beta_{k} = \frac{1}{k+2}$ made in \cref{Lpnv_RMK}, and \cref{summability_condition}.
    Then with probability $1$, $V_{k}$ converges to a nonnegative, finite-valued random variable.  
\end{theorem}
\begin{proof}
    Apply \cref{smct} and \cref{summability_lemma} with \[W^{k} = 2A_{k}\alpha_{k}R\Big{(} Var(z^{k}) + \frac{\mathbb{E}[Var(z^{k+1/2})|\bar{z}^{k},z^{k}])}{1 - \beta_{k}} \Big{)}, P^{k} = V_{k}, J^{k} = 0.\]
\end{proof}

Equipped with these results, we may now state the convergence of stochastic moving anchor methods.

\begin{theorem}[Stochastic Moving Anchor EAG-V Convergence]
    Consider the stochastic moving anchor EAG-V algorithm \eqref{stoch eagv update1}, \eqref{stoch eagv update2}, \eqref{stoch eagv update3} along with conditions \eqref{condition1}, \eqref{condition2}, \eqref{condition3}, the choice $\beta_{k} = \frac{1}{k+2}$ made in \cref{Lpnv_RMK}, and \cref{summability_condition} for the Lyapunov function \eqref{stoch_fxnal}.
    If $c_{\infty} \geq \frac{1}{\alpha_{\infty}}$, then the stochastic moving anchor EAG-V algorithm converges with rate 
    \[
        \|G(z^{k})\|^{2} \leq \frac{4\Big{[} (\alpha_{0}R^{2} + c_{0})\|z^{0} - z^{*}\|^{2} + \text{sum}(k-1) \Big{]}}{\alpha_{\infty}(k+1)(k+2)},
    \]
    where $\displaystyle \text{sum}(k-1) := \sum_{j=0}^{k-1} 2A_{j}\alpha_{j}R\big{(}Var(z^{j}) + \frac{\mathbb{E}[Var(z^{j+1/2})|\bar{z}^{j},z^{j}]}{1-\beta_{j}}\big{)}$.
\end{theorem}

\begin{proof}
    By \eqref{difference_bound}, we see that
    \begin{align}
        V_{k} &\leq V_{k-1} + 2A_{k-1}\alpha_{k-1}RVar(z^{k-1}) + \frac{2A_{k-1}}{1 - \beta_{k-1}}\alpha_{k-1}R\mathbb{E}[Var(z^{(k-1)+1/2})|\bar{z}^{k-1},z^{k-1}] \notag \\
        &\leq V_{0} + \text{sum}(k-1) \notag \\
        &= \alpha_{0}\|G(z^{0})\|^{2} + c_{0}\|z^{0} - z_{*}\|^{2} + \text{sum}(k-1) \notag \\
        &\leq (\alpha_{0}R + c_{0})\|z^{0} - z_{*}\|^{2} + \text{sum}(k-1). \notag
    \end{align}
    Going in the opposite direction, we see that
    \begin{align}
        \notag V_{k} &= A_{k}\|G(z^{k})\|^{2} + B_{k}\langle G(z^{k}),z^{k}-\bar{z}^{k}\rangle + c_{k}\|z^{*}-\bar{z}^{k}\|^{2}\\
        \notag &\geq A_{k}\|G(z^{k})\|^{2} + B_{k}\langle G(z^{k}),z^{*}-\bar{z}^{k}\rangle + c_{k}\|z^{*}-\bar{z}^{k}\|^{2} \text{ (monotonicity of $G$)}\\
        \notag &\geq \frac{A_{k}}{2}\|G(z^{k})\|^{2} + (c_{k} - \frac{B_{k}^{2}}{2A_{k}})\|z^{*}-\bar{z}^{k}\|^{2} \text{ (Young's inequality)}\\
        \notag &=\frac{\alpha_{k}(k+1)(k+2)}{4}\|G(z^{k})\|^{2} + (c_{k} - \frac{k+1}{\alpha_{k}(k+2)})\|z^{*}-\bar{z}^{k}\|^{2}\\
        \notag &\geq \frac{\alpha_{\infty}}{4}(k+1)(k+2)\|G(z^{k})\|^{2} + (c_{\infty} - \frac{1}{\alpha_{\infty}})\|z^{*}-\bar{z}^{k}\|^{2}\\
        \notag &\geq \frac{\alpha_{\infty}}{4}(k+1)(k+2)\|G(z^{k})\|^{2} \text{ ($c_{\infty}$ dominates $\frac{1}{\alpha_{\infty}}$)}
    \end{align}
    As long as $c_{\infty} \geq \frac{1}{\alpha_{\infty}},$ the second to last line above is positive, and we may focus on the inequality given to us by the last line above:
    \[
        \frac{\alpha_{\infty}}{4}(k+1)(k+2)\|G(z^{k})\|^{2} \leq (\alpha_{0}R^{2} + c_{0})\|z^{0}-z^{*}\|^{2} + sum(k-1).
    \]
    Finally, one divides both sides by the constant $\frac{\alpha_{\infty}}{4}(k+1)(k+2)$ to achieve the result.
\end{proof}

\section{Popov's Scheme with Moving Anchor} \label{sec: MAP_section}

In this section, we introduce an anchored Popov's scheme modified by the introduction of a moving anchor, an innovation first introduced in \cite{alcala2023moving} to improve numerical convergence results in \cite{yoon_ryu}, \cite{lee2021fast} and expand the scope of the extragradient-anchor theory.
This is directly inspired by the results in \cite{tran2021halpern}, and may also be considered an accelerated variant of the reflected gradient method \cite{Malitsky_2015}.

Before we proceed, let us first redirect our attention to the (fixed) anchored Popov scheme \cref{anch_pop_1}, \cref{anch_pop_2}.
In particular, note that $\hat{z}^{k}$ has two important roles: (1) it provides the direction of descent/ascent to obtain the next iterate $z^{k+1}$, and (2) it provides the direction of descent/ascent for the next extrapolation, $\hat{z}^{k+1}$.
This subtle distinction matters, because in the moving anchor setups \cite{alcala2023moving}, the descent directions are simply computed sequentially: whatever point you most recently computed is going to be used in your next computation, including the moving anchor step.
Now that that directions are \textit{only} computed via the extrapolators $\hat{z}^{k}$, we must consider the fact that the moving anchor was computed with a descent/ascent step whose direction came from the \textit{most recently computed iterate}, $z^{k+1}$.
In the algorithms developed and studied in \cite{alcala2023moving}, $z^{k+1}$ is then used to compute the next extrapolation point and the next iterate.
By construction, this point cannot be used as a descent direction for the primary steps in any Popov's scheme variant with a moving anchor.

With this discussion, there are two choices for the moving anchor descent direction in a Popov's scheme with moving anchor:
\begin{enumerate}
    \item $z^{k+1}$ \label{choice_1}
    \item $\hat{z}^{k}$. \label{choice_2}
\end{enumerate}
The advantage to \cref{choice_1} is that it follows the practice established in \cite{alcala2023moving} of using the most recently computed iterate, and can therefore potentially take advantage of the convergence theory therein established. 
On the other hand, \cref{choice_2} may also be considered a natural choice, as it follows the Popov's scheme practice of only computing descent directions from extrapolators and benefits the analysis in only needing to worry about handling gradients on such points.

In light of this discussion and the numerical results in \cref{sec: num_exp}, we offer two Popov's schemes with moving anchors.

Popov's scheme with Moving Anchor Version 1: Last Iterate Direction

$\begin{cases}
    \hat{z}^{k} &= \beta_{k}\bar{z}^{k} + (1 - \beta_{k})z^{k} - \eta_{k}G(\hat{z}^{k-1}) \\
    z^{k+1} &= \beta_{k}\bar{z}^{k} + (1 - \beta_{k})z^{k} - \eta_{k}G(\hat{z}^{k}) \label{version_1}\\
    \bar{z}^{k+1} &= \bar{z}^{k} \pm \gamma_{k+1}G(z^{k+1})
\end{cases}$

Popov's scheme with Moving Anchor Version 2: Extrapolator Direction

$\begin{cases}
    \hat{z}^{k} &= \beta_{k}\bar{z}^{k} + (1 - \beta_{k})z^{k} - \eta_{k}G(\hat{z}^{k-1}) \\
    z^{k+1} &= \beta_{k}\bar{z}^{k} + (1 - \beta_{k})z^{k} - \eta_{k}G(\hat{z}^{k}) \label{version_2}\\
    \bar{z}^{k+1} &= \bar{z}^{k} \pm \gamma_{k+1}G(\hat{z}^{k})
\end{cases}$

\subsection{Discussion}

Preliminary attempts at analyzing Version 1 with a moving anchor structure otherwise unchanged from that in \cite{alcala2023moving} suggest that there are no issues in obtaining a convergence theorem similar to the one found in \cite{tran2021halpern} with a fixed anchor structure, but there are still some issues of absolving certain terms when it comes to constructing a Lyapunov descent lemma.

On the other hand, numerical results and the existing Popov schemes handling descent directions via extrapolation points indicate that a convergence analysis on Version 2 may also be fruitful; we leave the completion of this analysis and the exploration of other anchor step descent directions, such as a convex combination of $z^{k+1}$ and $\hat{z}^{k}$, as future work.

While the convergence theory here is incomplete, our tentative numerical results in \cref{sec: num_exp} indicate that such a theory is plausible and may include both Version 1 and Version 2 of the Popov's scheme with moving anchor.

\section{Numerical Experiments} \label{sec: num_exp}

\subsection{Stochastic Experiments}

Here we detail multiple numerical experiments that showcase the stochastic moving anchor. The choices for some initialization constants must differ significantly from their deterministic counterparts, in particular:

\begin{align*}
    \alpha_{0} &= 0.9\cdot(3/4)\cdot(1/R) \cdot (1 / \sqrt{K_{G}})\\
    c_{0} &= (1.01)\cdot (4/3) \cdot e^{\frac{\pi^{2}}{6}} \cdot R\sqrt{K_{G}}
\end{align*}

\begin{figure}[h!]
	\centering
	\includegraphics[scale=.5]{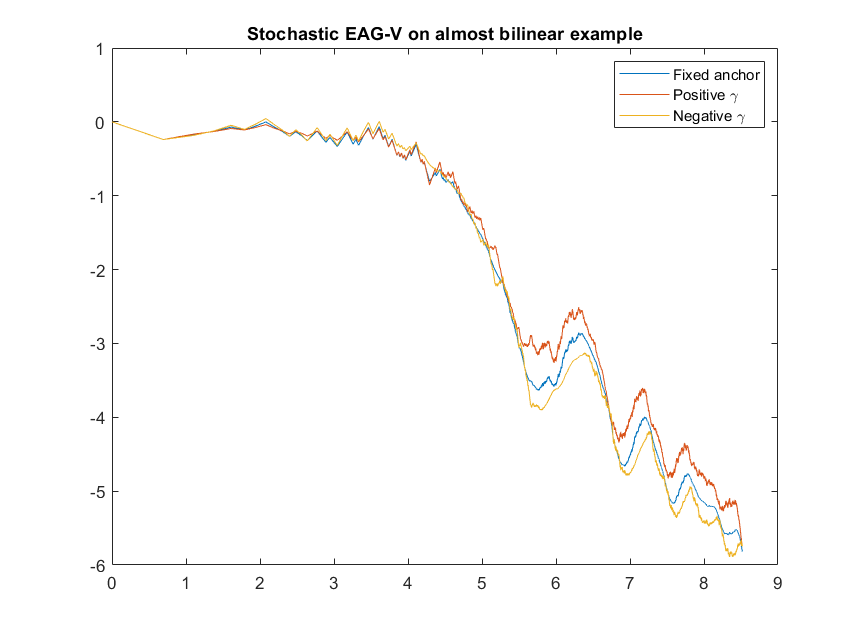}
	\caption{Comparison of the grad-norm squared of three stochastic EAG-V variants of interest on a toy `almost bilinear' problem.}
	\label{fig:grad_norms_sqrd_stoch_eagv}
\end{figure}

In \cref{fig:grad_norms_sqrd_stoch_eagv}, one observes that, on a toy example, the behavior of the $-\gamma_{k}$ and $+\gamma_{k}$ variants of the moving anchor seem to parallel the deterministic setting of the same problem \cite{alcala2023moving}, in that the negative version is the fastest and the positive version is the slowest.


\begin{figure}[h!]
	\centering
	\includegraphics[scale=.5]{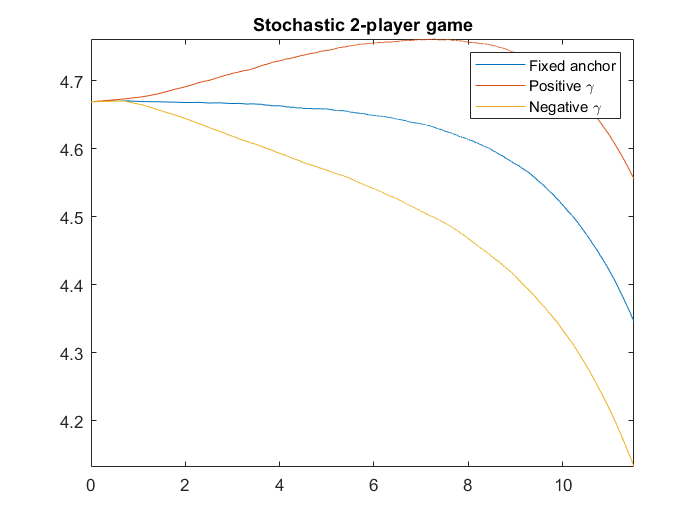}
	\caption{Comparison of the grad-norm squared of three stochastic EAG-V anchoring variants on a nonlinear game.}
    \label{fig:stoch_game}
\end{figure}

In \cref{fig:stoch_game}, a 2 player nonlinear game of the following form was studied using the stochastic moving anchor EAG-V algorithms:

\[\min_{x \in \Delta^{n}} \max_{y \in \Delta^{m}} \frac{1}{2}\langle Qx,x\rangle + \langle Kx,y \rangle \]
where $Q = A^{T}A$ is positive semidefinite for $A \in \mathbb{R}^{k \times n}$ which has entries generated independently from the standard normal distribution, $K \in \mathbb{R}^{m \times n}$ with entries generated uniformly and independently from the interval $[-1,1],$ and $\Delta^{n}, \Delta^{m}$ are the $n-$ and $m-$simplices, respectively:
\[\Delta^{n} := \Big{\{} x\in\mathbb{R}_{+}^{n}: \sum_{i=1}^{n}x_{i} = 1 \Big{\}}, \; \Delta^{m} := \Big{\{} y\in\mathbb{R}_{+}^{m}: \sum_{j=1}^{m}y_{j} = 1 \Big{\}}.\]
One may interpret this as a two person game where player one has $n$ strategies to choose from, choosing strategy $i$ with probability $x_{i} \; (i = 1, . . . , n)$ to attempt to minimize a loss, while the second player attempt to maximize their gain among $m$ strategies with strategy $j$ chosen with probability $y_{j} \; (j = 1, \ldots , m).$
The payoff is a quadratic function that depends on the strategy of both players.
This was implemented following the 3 operator splitting scheme in \cite{davis2017three}, with parameter $\lambda = 0.01$.
We set $m = 2$ and $n=48,$ for a more moderately sized problem with a well-behaved condition number.

A few remarks are in order.
This game was previously studied in both \cite{chen2014optimal} and \cite{alcala2023moving}, where in the latter case the authors encountered favorable results using deterministic moving anchor algorithms.
However, in both high and low dimensional examples, the choice of positive $\gamma_{k}$ resulted in the most significant acceleration beating the fixed anchor in the deterministic case.
Here in our stochastic variant, we encounter the opposite behavior: a moving anchor variant indeed provides the most significant acceleration, but it is the negative $\gamma_{k}$ that does so.
One possibility is that the EAG-V algorithm structure somehow favors the negative $\gamma_{k}$ variant of moving anchor algorithms, while the FEG algorithm structure - whether on a convex-concave problem or a nonconvex-nonconcave problem - favors the positive $\gamma_{k}$.
Secondly, although our theory calls for setting $K_{G}$ to be the operator's condition number times the dimensions in the domain of the objective function, setting $K_{G} = 1$ resulted in significant numerical improvements, especially as even with $m = 2, n = 48,$ the condition number in this type of problem can be very large.
Finally, the choice of the parameter $\lambda$ relating to the three operator splitting structure differs a large amount from that used in \cite{alcala2023moving}, and it is unclear why this numerical discrepancy exists.

\subsection{Anchored Popov's Scheme}

\begin{figure}[h!]
	\centering
	\includegraphics[scale=.5]{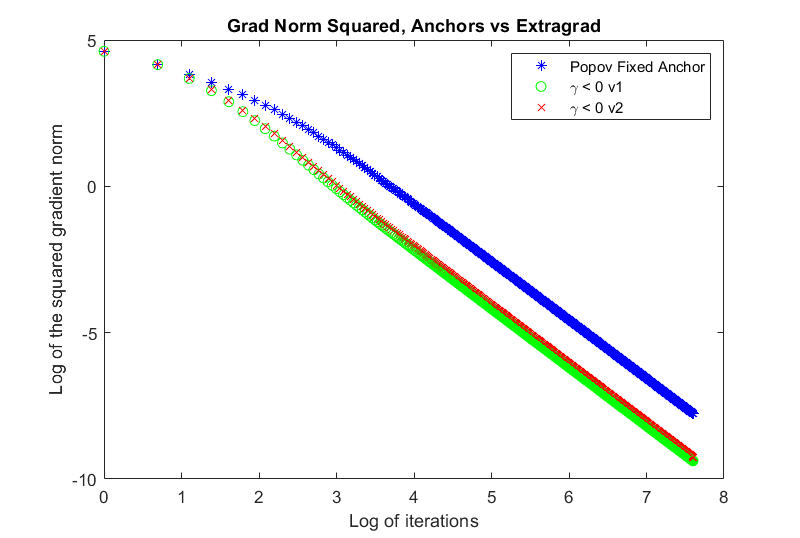}
	\caption{Comparison of the fixed anchor Popov scheme versus two moving anchor variants with a descending anchor.}
    \label{fig:popov_example}
\end{figure}

\cref{fig:popov_example} shows the original anchored Popov's scheme \cite{tran2021halpern} compared to two moving anchor Popov's schemes, see \cref{sec: MAP_section} for details on the algorithm construction.
The problem studied is the toy `almost bilinear' problem $f(x,y) = \epsilon \frac{\|x\|^{2}}{2} + \langle x,y \rangle - \epsilon \frac{\|y\|^{2}}{2}$ in one dimension, this time with $\epsilon$ tuned up to $10^{1}$.
As in \cite{alcala2023moving}, we see that two moving anchor variants are faster than the fixed anchor variant, this time in the anchored Popov's scheme setup.
It also appears that this setup mirrors that of the moving anchor EAG-V studied in \cite{alcala2023moving}, as we see that the negative $\gamma_{k}$ variants of the moving anchor algorithm are the fastest by a constant.
Furthermore, both version 1 and version 2 of the moving anchor Popov schemes with a negative $\gamma_{k}$ are faster by a constant, implying that not only is there an uncovered convergence theory for the moving anchor in this setup, but that it may somehow incorporate both versions 1 and 2 studied in this example.
We leave more numerical examples as well as a fully rigorous convergence theory for future work.

\section{Conclusion}

In this work we introduce stochasticity to the previously developed moving anchor methods \cite{alcala2023moving} with both rigorous theory and numerical results demonstrating their efficacy. A Lyapunov analysis enables one to demonstrate order-optimal convergence results with no hindrance to the strength of the algorithms in computational examples. In addition, a moving anchor Popov's scheme is developed following \cite{tran2021halpern} that, while currently lacking rigorous convergence theory, shows promise with numerical results hinting at a convergence theory that we hope to uncover soon. In particular, because of the differences in how Popov's scheme and previous moving anchor methods compute descent directions for their anchors, there may be a convergence theory that allows for a wide range of different descent directions, which was previously unheard of in both fixed and moving anchor methods. 
More computational examples and a convergence theory for the moving anchor Popov's scheme are warranted.

\bibliographystyle{amsplain}
\bibliography{bibliography}

\end{document}